\documentclass[letterpaper, 10 pt, conference]{ieeeconf}  

\IEEEoverridecommandlockouts                              
\usepackage{times,latexsym,amssymb,amsmath,theorem,epsfig}
\usepackage{mathptm}
\usepackage[hyperindex]{hyperref}

\newtheorem{theorem}{Theorem}[section]

\newtheorem{definition}[theorem]{Definition}
{\theorembodyfont{\rmfamily} \newtheorem{example}[theorem]{Example}}

\newtheorem{problem}[theorem]{Problem}

\newtheorem{proposition}[theorem]{Proposition}

{\theorembodyfont{\rmfamily} }

\renewcommand{\QED}{\QEDopen}

\title{\LARGE \bf An Observable Canonical Form\\
for a Rational System on a Variety}

\author{Jana N\v{e}mcov\'{a}$^{1}$ and Jan H. van Schuppen$^{2}$%
\thanks{The support of the University of Chemistry and Technology
  for the cooperation of the two authors is herewith gratefully acknowledged.}%
\thanks{$^{1}$Jana N\v{e}mcov\'{a} is with the Department of Mathematics,
  University of Chemistry and Technology, Technik\'{a} 5, 166 28 Prague 6, Czech Republic
  {\tt\small Jana.Nemcova@vscht.cz}} \\
\thanks{$^{2}$ Jan H. van Schuppen is with
  Van Schuppen Control Re\-se\-arch,
  Gouden Leeuw 143, 1103 KB Amsterdam, The Netherlands
  {\tt\small jan.h.van.schuppen@xs4all.nl}}%
}

\begin{document}

\newcommand{\sign}{\mbox{sign}}

\newcommand{\mathbbd}{\mbox{${\mathbb{D}}$}} 
\newcommand{\DS}{\mbox{${\mathbb{DS}}$}} 
\newcommand{\DSC}{\mbox{${\mathbb{DSC}}$}} 
\newcommand{\M}{\mbox{${\mathbb{M}}$}} 
\newcommand{\N}{\mbox{${\mathbb{N}}$}} 
\newcommand{\PP}{\mbox{${\mathbb{P}}$}} 
\newcommand{\Q}{\mbox{${\mathbb{Q}}$}} 
\newcommand{\W}{\mbox{${\mathbb{W}}$}} 
\newcommand{\C}{\mbox{$C$}} 
\newcommand{\R}{\mbox{$R$}} 
\newcommand{\T}{\mbox{$T$}} 
\newcommand{\Z}{\mbox{$Z$}}

\newcommand{\boundary}{\mbox{$\rm bndry$}}
\newcommand{\card}{\mbox{$\rm card$}}
\newcommand{\cisetsmin}{\mbox{$\rm CISets_{\min}$}}
\newcommand{\cisets}{\mbox{$\rm CISets$}}
\newcommand{\cix}{\mbox{$\rm CIX$}}
\newcommand{\colpwr}{\mbox{Colpwr}}
\newcommand{\closure}{\mbox{cl}}
\newcommand{\da}{\mbox{$\rm DA$}}
\newcommand{\domainofattraction}{\mbox{$DA$}}
\newcommand{\ds}{\mbox{$\rm DS$}}
\newcommand{\graph}{\mbox{Graph}}
\newcommand{\image}{\mbox{Img}}
\newcommand{\init}{\mbox{Init}}
\newcommand{\interior}{\mbox{int}}
\newcommand{\Mp}{\mbox{\bf M}}
\newcommand{\powerset}{\mbox{Pwrset}}
\newcommand{\pwr}{\mbox{Pwrset}}

\newcommand{\orderlexg}{\mbox{$>_{{\rm lex}}$}}
\newcommand{\orderlexl}{\mbox{$<_{{\rm lex}}$}}
\newcommand{\nm}{\mbox{$\mathbb{N}_m$}}
\newcommand{\relationconnect}{\mbox{${\rm R_{connect}}$}}
\newcommand{\zk}{\mbox{$\mathbb{Z}_k$}}
\newcommand{\zn}{\mbox{$\mathbb{Z}_n$}}
\newcommand{\zm}{\mbox{$\mathbb{Z}_m$}}
\newcommand{\zo}{\mbox{${\rm O}$}} 
\newcommand{\zr}{\mbox{$\mathbb{Z}_r$}}
\newcommand{\zp}{\mbox{$\mathbb{Z}_p$}}
\newcommand{\zpos}{\mbox{$\mathbb{Z}_{+}$}}
\newcommand{\zposn}{\mbox{$\mathbb{Z}_{+}^n$}}

\newcommand{\almosteverywhere}{\mbox{$\forall E$}}
\newcommand{\realnumbersextended}{\mbox{$\overline{\mathbb{R}}$}}
\newcommand{\real}{\mbox{$\mathbb{R}$}}
\newcommand{\reals}{\mbox{$\mathbb{R}$}}

\newcommand{\cdiscclosed}{\mbox{$\mathbb{D}_c$}}
\newcommand{\cdiscopen}{\mbox{$\mathbb{D}_o$}}
\newcommand{\cdiscoutopen}{\mbox{$(\mathbb{D}^c)_o$}}
\newcommand{\cminus}{\mbox{$\mathbb{C}^-$}}
\newcommand{\cplus}{\mbox{$\mathbb{C}^+$}}
\newcommand{\im}{\mbox{\rm Im}}
\newcommand{\re}{\mbox{\rm Re}}

\newcommand{\cilspaces}{\mbox{${\rm CILSpaces}$}} 
\newcommand{\cilssp}{\mbox{${\rm CILSP}$}} 
\newcommand{\cilsp}{\mbox{${\rm CILSp}$}} 
\newcommand{\cli}{\mbox{${\rm CLI}$}} 
\newcommand{\climin}{\mbox{${\rm CLI_{min}}$}} 
\newcommand{\cn}{\mbox{$\mathbb{C}^n$}} 
\newcommand{\cmcm}{\mbox{$\mathbb{C}^m$}} 
\newcommand{\linearsubspaces}{\mbox{${\rm LinSubspaces}$}} 
\newcommand{\lat}{\mbox{${\rm Lat}$}} 
\newcommand{\quotientspaces}{\mbox{${\rm Quotientspaces}$}} 
\newcommand{\rinfty}{\mbox{$\mathbb{R}^{\infty}$}}
\newcommand{\rk}{\mbox{$\mathbb{R}^k$}}
\newcommand{\rmrm}{\mbox{$\mathbb{R}^m$}}
\newcommand{\rn}{\mbox{$\mathbb{R}^n$}} 
\newcommand{\rp}{\mbox{$\mathbb{R}^p$}}
\newcommand{\rpos}{\mbox{$\mathbb{R}_{+}$}}
\newcommand{\rposk}{\mbox{$\mathbb{R}_{+}^k$}}
\newcommand{\rposm}{\mbox{$\mathbb{R}_{+}^m$}}
\newcommand{\rposn}{\mbox{$\mathbb{R}_{+}^n$}}
\newcommand{\rposp}{\mbox{$\mathbb{R}_{+}^p$}}
\newcommand{\rr}{\mbox{$\mathbb{R}^r$}}
\newcommand{\rspos}{\mbox{$\mathbb{R}_{s+}$}}
\newcommand{\rsposn}{\mbox{$\mathbb{R}_{s+}^n$}}
\newcommand{\simplexn}{\mbox{$\mathbb{S}_{+}^n$}}
\newcommand{\sn}{\mbox{$\mathbb{S}_{+}^n$}}

\newcommand{\cimod}{\mbox{${\rm CIModules}$}} 

\newcommand{\switch}{\mbox{${\rm Switch}$}} 
\newcommand{\threshold}{\mbox{${\rm Threshold}$}}

\newcommand{\ara}{\mbox{${\rm ARA}$}} 
\newcommand{\dom}{\mbox{\rm Dom}}
\newcommand{\rdom}{\mbox{$\rm Dom$}}
\newcommand{\Fc}{\mbox{${\bf F}$}} 
\newcommand{\hara}{\mbox{${\rm HARA}$}} 
\newcommand{\range}{\mbox{\rm Range}}
\newcommand{\rra}{\mbox{${\rm RRA}$}} 
\newcommand{\U}{\mbox{${\rm U}$}} 
\newcommand{\A}{\mbox{${\rm A}$}} 

\newcommand{\fanal}{\mathcal{C}^{\omega}}
\newcommand{\fsmooth}{\mathcal{C}^{\infty}}
\newcommand{\fcont}{\mathcal{C}}
\newcommand{\fpcont}{\mathcal{PC}}
\newcommand{\fpconst}{\mathcal{PC}onst}

\newcommand{\arrow}{\mbox{${\rm arrow}$}}
\newcommand{\blockdiagonal}{\mbox{${\rm Bdiag}$}}
\newcommand{\cnm}{\mbox{$\mathbb{C}^{n \times m}$}}
\newcommand{\cnn}{\mbox{$\mathbb{C}^{n \times n}$}}
\newcommand{\cnp}{\mbox{$\mathbb{C}^{n \times p}$}}
\newcommand{\cpn}{\mbox{$\mathbb{C}^{p \times n}$}}
\newcommand{\cpp}{\mbox{$\mathbb{C}^{p \times p}$}}
\newcommand{\cum}{\mbox{${\rm cum}$}}
\newcommand{\dnn}{\mbox{$\mathbb{R}_{\rm diag}^{n \times n}$}}
\newcommand{\dnnpos}{\mbox{$\mathbb{R}_{\rm +, diag}^{n \times n}$}}
\newcommand{\dnnspos}{\mbox{$\mathbb{R}_{\rm diag, +}^{n \times n}$}}
\newcommand{\imp}{\mbox{${\rm Imprim}$}}
\newcommand{\invertible}{\mbox{${\rm inv}$}}
\newcommand{\rkk}{\mbox{$\mathbb{R}^{k \times k}$}}
\newcommand{\rkm}{\mbox{$\mathbb{R}^{k \times m}$}}
\newcommand{\rkn}{\mbox{$\mathbb{R}^{k \times n}$}}
\newcommand{\rmk}{\mbox{$\mathbb{R}^{m \times k}$}}
\newcommand{\rmm}{\mbox{$\mathbb{R}^{m \times m}$}}
\newcommand{\rmn}{\mbox{$\mathbb{R}^{m \times n}$}}
\newcommand{\rmnarrow}{\mbox{$\mathbb{R}_{\arrow}^{m \times n}$}}
\newcommand{\rmnblockdiagonal}{\mbox{$\mathbb{R}_{{\rm Bdiag}}^{m \times n}$}}
\newcommand{\rmncoordinated}{\mbox{$\mathbb{R}_{{\rm c}}^{m \times n}$}}
\newcommand{\rmr}{\mbox{$\mathbb{R}^{m \times r}$}}
\newcommand{\rnk}{\mbox{$\mathbb{R}^{n \times k}$}}
\newcommand{\rnm}{\mbox{$\mathbb{R}^{n \times m}$}}
\newcommand{\rnn}{\mbox{$\mathbb{R}^{n \times n}$}}
\newcommand{\rnnarrow}{\mbox{$\mathbb{R}_{\arrow}^{n \times n}$}}
\newcommand{\rnncoordinated}{\mbox{$\mathbb{R}_{{\rm c}}^{n \times n}$}}
\newcommand{\rnnspd}{\mbox{$\mathbb{R}_{spd}^{n \times n}$}}
\newcommand{\rnnsspd}{\mbox{$\mathbb{R}_{sspd}^{n \times n}$}}
\newcommand{\rnp}{\mbox{$\mathbb{R}^{n \times p}$}}
\newcommand{\rnr}{\mbox{$\mathbb{R}^{n \times r}$}}
\newcommand{\rpk}{\mbox{$\mathbb{R}^{p \times k}$}}
\newcommand{\rpm}{\mbox{$\mathbb{R}^{p \times m}$}}
\newcommand{\rpn}{\mbox{$\mathbb{R}^{p \times n}$}}
\newcommand{\rpp}{\mbox{$\mathbb{R}^{p \times p}$}}
\newcommand{\rpr}{\mbox{$\mathbb{R}^{p \times r}$}}
\newcommand{\rrn}{\mbox{$\mathbb{R}^{r \times n}$}}
\newcommand{\rrr}{\mbox{$\mathbb{R}^{r \times r}$}}
\newcommand{\sposdefnn}{\mbox{$\mathbb{RSPD}^{n \times n}$}}
\newcommand{\spd}{\mbox{$\mathbb{R}_{\rm spd}^{n \times n}$}}

\newcommand{\dposnn}{\mbox{$\mathbb{R}_{+,{\rm diag}}^{n \times n}$}}
\newcommand{\dpossnn}{\mbox{$\mathbb{R}_{s+,{\rm diag}}^{n \times n}$}}
\newcommand{\dsposnn}{\mbox{$\mathbb{R}_{s+,{\rm diag}}^{n \times n}$}}
\newcommand{\dscnn}{\mbox{$\mathbb{DSC}_{+}^{n \times n}$}}
\newcommand{\dsnn}{\mbox{$\mathbb{DS}_{+}^{n \times n}$}}
\newcommand{\indeximprim}{\mbox{\rm imprim}}
\newcommand{\indexprim}{\mbox{\rm indexprim}}
\newcommand{\imprim}{\mbox{\rm imprim}}
\newcommand{\matrixrowtrunc}{\mbox{$\rm matrowtrunc$}}
\newcommand{\matrixdiagtrunc}{\mbox{$\rm mdtrunc$}}
\newcommand{\mnn}{\mbox{$\mathbb{M}_{+}^{n \times n}$}}
\newcommand{\mposkk}{\mbox{$M_{+}^{k \times k}$}}
\newcommand{\mposmm}{\mbox{$M_{+}^{m \times m}$}}
\newcommand{\mposnn}{\mbox{$M_{+}^{n \times n}$}}
\newcommand{\nn}{\mbox{$\mathbb{N}_n$}}
\newcommand{\permnn}{\mbox{$\mathbb{P}^{n \times n}$}}
\newcommand{\permpp}{\mbox{$\mathbb{P}^{p \times p}$}}
\newcommand{\posR}[2]{\mbox{$R_{+}^{#1\times#2}$}}
\newcommand{\rposik}{\mbox{$\mathbb{R}_{+}^{\infty\times k}$}}
\newcommand{\rposim}{\mbox{$\mathbb{R}_{+}^{\infty\times m}$}}
\newcommand{\rposin}{\mbox{$\mathbb{R}_{+}^{\infty\times n}$}}
\newcommand{\rposiq}{\mbox{$\mathbb{R}_{+}^{\infty\times q}$}}
\newcommand{\rposkk}{\mbox{$\mathbb{R}_{+}^{k \times k}$}}
\newcommand{\rposkm}{\mbox{$\mathbb{R}_{+}^{k \times m}$}}
\newcommand{\rposkn}{\mbox{$\mathbb{R}_{+}^{k \times n}$}}
\newcommand{\rposkp}{\mbox{$\mathbb{R}_{+}^{k \times p}$}}
\newcommand{\rposmm}{\mbox{$\mathbb{R}_{+}^{m \times m}$}}
\newcommand{\rposmk}{\mbox{$\mathbb{R}_{+}^{m \times k}$}}
\newcommand{\rposmn}{\mbox{$\mathbb{R}_{+}^{m \times n}$}}
\newcommand{\rposnk}{\mbox{$\mathbb{R}_{+}^{n \times k}$}}
\newcommand{\rposnm}{\mbox{$\mathbb{R}_{+}^{n \times m}$}}
\newcommand{\rposnn}{\mbox{$\mathbb{R}_{+}^{n \times n}$}}
\newcommand{\rposnp}{\mbox{$\mathbb{R}_{+}^{n \times p}$}}
\newcommand{\rpospm}{\mbox{$\mathbb{R}_{+}^{p \times m}$}}
\newcommand{\rpospn}{\mbox{$\mathbb{R}_{+}^{p \times n}$}}
\newcommand{\rposqm}{\mbox{$\mathbb{R}_{+}^{q \times m}$}}
\newcommand{\rposqn}{\mbox{$\mathbb{R}_{+}^{q \times n}$}}
\newcommand{\rposqq}{\mbox{$\mathbb{R}_{+}^{q \times q}$}}
\newcommand{\rsposnn}{\mbox{$\mathbb{R}_{s+}^{n \times n}$}}
\newcommand{\Table}{\mbox{$\rm table$}}
\newcommand{\wnn}{\mbox{$\W^{n \times n}$}}

\newcommand{\adjoint}{\mbox{$\rm Adj$}}
\newcommand{\affine}{\mbox{$\rm Affine$}}
\newcommand{\affinedim}{\mbox{$\rm Affdim$}}
\newcommand{\affineb}{\mbox{$\rm AffB$}}
\newcommand{\blockdiag}{\mbox{{\rm Block-diag}}}
\newcommand{\circulant}{\mbox{$\rm Circulant$}}
\newcommand{\cols}{\mbox{$\rm col$}}
\newcommand{\diag}{\mbox{$\rm Diag$}}
\newcommand{\dsc}{doubly stochastic circulant }
\newcommand{\dscs}{doubly stochastic circulants }
\newcommand{\eig}{\mbox{$\rm eig$}}
\newcommand{\posr}[1]{#1\mbox{\rm -pos-rank}}
\newcommand{\posrank}{\mbox{$\rm pos-rank$}}
\newcommand{\projection}{\mbox{$\rm proj$}}
\newcommand{\rank}{\mbox{$\rm rank$}}
\newcommand{\rowrankz}{\mbox{$\rm row-rank_Z$}}
\newcommand{\lspan}{\mbox{$\rm span$}}
\newcommand{\specrad}{\mbox{$\rm specrad$}}
\newcommand{\spec}{\mbox{$\rm spec$}}
\newcommand{\spectrum}{\mbox{$\rm spec$}}
\newcommand{\tr}{\mbox{$\rm tr$}}
\newcommand{\trace}{\mbox{\rm trace}}

\newcommand{\affinehull}{\mbox{$\rm affh$}}
\newcommand{\convexhull}{\mbox{$\rm convh$}}
\newcommand{\ri}{\mbox{\rm ri}}
\newcommand{\triangulation}{\mbox{\rm Tr}}

\newcommand{\cekk}{\mbox{$CE_{k,k}$}}
\newcommand{\cekm}{\mbox{$CE_{k,m}$}}
\newcommand{\cekn}{\mbox{$CE_{k,n}$}}
\newcommand{\cenn}{\mbox{$CE_{n,n}$}}
\newcommand{\ckk}{\mbox{$\mathbb{C}_{k,k}$}}
\newcommand{\ckm}{\mbox{$\mathbb{C}_{k,m}$}}
\newcommand{\cone}{\mbox{\rm cone}} 
\newcommand{\faces}{\mbox{\rm Faces}}
\newcommand{\hp}{\mbox{\rm HyperPlane}}
\newcommand{\normalvector}{\mbox{$V_{normal}$}}
\newcommand{\phs}{\mbox{$PHS$}}
\newcommand{\plsets}{\mbox{$PLSets$}}
\newcommand{\polyhedralcone}{\mbox{\rm Polyhcone}}
\newcommand{\shp}{\mbox{\rm SupportHyperPlane}}
\newcommand{\subpolytope}{\mbox{\rm SubPolytope}}
\newcommand{\subrectangle}{\mbox{\rm SubRectangle}}
\newcommand{\vertices}{\mbox{$V_{vertices}$}}

\newcommand{\qy}{\mbox{$Q_y$}}
\newcommand{\qlsdp}{\mbox{${\bf \partial Q_{lsdp}}$}}
\newcommand{\qp}{\mbox{${\bf Q_{lsp}}$}}
\newcommand{\qpd}{\mbox{${\bf Q_{lsdp}}$}}
\newcommand{\qpr}{\mbox{${\bf Q_{lsp,r}}$}}
\newcommand{\qprs}{\mbox{${\bf Q_{lsp,s}}$}}
\newcommand{\qpdr}{\mbox{${\bf Q_{lsdp,r}}$}}
\newcommand{\dqps}{\mbox{${\bf \partial Q_{lsp,r}}$}}
\newcommand{\dqpss}{\mbox{${\bf \partial Q_{lsp,s}}$}}
\newcommand{\dqpdr}{\mbox{${\bf \partial Q_{lsdp,r}}$}}
\newcommand{\dqprs}{\mbox{${\bf \partial Q_{lsp,r,s}}$}}
\newcommand{\dqpdrs}{\mbox{${\bf \partial Q_{lsdp,r,s}}$}}
\newcommand{\dqp}{\mbox{${\bf \partial Q_{lsp}}$}}
\newcommand{\dqpd}{\mbox{${\bf \partial Q_{lsdp}}$}}

\newcommand{\con}{\mbox{${\rm con}$}}
\newcommand{\conmat}{\mbox{${\rm conmat}$}}
\newcommand{\conset}{\mbox{${\rm Conset}$}}
\newcommand{\coconset}{\mbox{${\rm co-Conset}$}}
\newcommand{\controllablepair}{\mbox{${\rm conpair}$}}
\newcommand{\controllabilitymatrix}{\mbox{${\rm conmat}$}}
\newcommand{\controllableset}{\mbox{${\rm conset}$}}
\newcommand{\cocontrollableset}{\mbox{${\rm co-conset}$}}
\newcommand{\controlset}{\mbox{$\rm controlset$}}
\newcommand{\ls}{\mbox{${\rm LS}$}}
\newcommand{\lsp}{\mbox{${\rm LSP}$}}
\newcommand{\lspmin}{\mbox{${\rm LSP_{min}}$}}
\newcommand{\reachm}{\mbox{${\rm reachm}$}}
\newcommand{\obsm}{\mbox{${\rm obsm}$}}
\newcommand{\obsmat}{\mbox{${\rm obsmat}$}}
\newcommand{\obsmap}{\mbox{${\rm obsmap}$}}
\newcommand{\realization}{\mbox{${\rm realiz}$}}
\newcommand{\reconmap}{\mbox{${\rm reconmap}$}}

\newcommand{\argmin}{\mbox{${\rm argmin}$}}
\newcommand{\argmax}{\mbox{${\rm argmax}$}}
\newcommand{\rai}{\mbox{${\rm RAI}$}}

\newcommand{\aobs}{\mbox{${\rm A_{obs}}$}}
\newcommand{\bdinf}{\mbox{${\rm bdinf}$}}
\newcommand{\bdjac}{\mbox{${\rm bdjac}$}}
\newcommand{\bdsup}{\mbox{${\rm bdsup}$}}
\newcommand{\can}{\mbox{${\rm can}$}}
\newcommand{\obs}{\mbox{${\rm obs}$}}
\newcommand{\qobs}{\mbox{${\rm Q_{obs}}$}}
\newcommand{\svdtruncation}{\mbox{${\rm SVDtrunc}$}}
\newcommand{\zclosure}{\mbox{${\rm \mathcal{Z}-cl}$}}

\newcommand{\act}{\mbox{$\rm act$}}
\newcommand{\aux}{\mbox{$\rm Aux$}}
\newcommand{\cat}{\mbox{$\rm cat$}}
\newcommand{\cigdes}{\mbox{$\rm CIG$}}
\newcommand{\cil}{\mbox{$\rm CIL$}}
\newcommand{\child}{\mbox{$\rm Chi$}}
\newcommand{\co}{\mbox{$\rm CO$}}
\newcommand{\codes}{\mbox{$\rm CODES$}}
\newcommand{\controllabletriple}{\mbox{$\rm ConTriple$}}
\newcommand{\coreach}{\mbox{$\rm coreach$}}
\newcommand{\coreachcomponent}{\mbox{$\rm coreachco$}}
\newcommand{\coreachgen}{\mbox{$\rm coreachgen$}}
\newcommand{\coreachset}{\mbox{$\rm coreachset$}}
\newcommand{\csublanguage}{\mbox{$\rm C$}}
\newcommand{\cpcsublanguage}{\mbox{$\rm C_{pc}$}}
\newcommand{\csuplanguage}{\mbox{$\rm CSupL$}}
\newcommand{\cpcsuplanguage}{\mbox{$\rm CSupL_{pc}$}}
\newcommand{\decdes}{\mbox{$\rm decDES$}}
\newcommand{\infcsuplanguage}{\mbox{$\inf {\rm CSupL}$}}
\newcommand{\infcsuppclanguage}{\mbox{$\inf {\rm CSupL_{pc}}$}}
\newcommand{\last}{\mbox{$\rm last$}}
\newcommand{\length}{\mbox{$\rm length$}}
\newcommand{\markact}{\mbox{$\rm markact$}}
\newcommand{\moddes}{\mbox{$\rm modDES$}}
\newcommand{\nextact}{\mbox{$\rm nextact$}}
\newcommand{\normaltuple}{\mbox{$\rm NormalTuple$}}
\newcommand{\parent}{\mbox{$\rm Par$}}
\newcommand{\prefix}{\mbox{$\rm prefix$}}
\newcommand{\ps}{\mbox{$\rm PS$}}
\newcommand{\qeq}{\mbox{$\rm QEQ$}}
\newcommand{\reachcomponent}{\mbox{$\rm reachco$}}
\newcommand{\reachgen}{\mbox{$\rm reachgen$}}
\newcommand{\reachset}{\mbox{$\rm reachset$}}
\newcommand{\runs}{\mbox{$\rm Runs$}}
\newcommand{\selfloop}{\mbox{$\rm selfloop$}}
\newcommand{\setldfg}{\mbox{$\rm SetL_{DFG}$}}
\newcommand{\setllc}{\mbox{$\rm SetL_{lc}$}}
\newcommand{\setllcn}{\mbox{$\rm SetL_{lcn}$}}
\newcommand{\setln}{\mbox{$\rm SetL_N$}}
\newcommand{\setlnfg}{\mbox{$\rm SetL_{NFG}$}}
\newcommand{\setlreg}{\mbox{$\rm SetL_{reg}$}}
\newcommand{\suffix}{\mbox{$\rm suffix$}}
\newcommand{\shuffle}{\mbox{$\rm shuffle$}}
\newcommand{\size}{\mbox{$\rm Size$}}
\newcommand{\starrr}{\mbox{$\rm Star$}}
\newcommand{\supap}{\mbox{$\sup {\rm AP}$}}
\newcommand{\supc}{\mbox{$\sup {\rm C}$}}
\newcommand{\supcc}{\mbox{$\sup {\rm cC}$}}
\newcommand{\supccn}{\mbox{$\sup {\rm cCN}$}}
\newcommand{\supcsublanguage}{\mbox{$\sup {\rm C}$}}
\newcommand{\supcpcsublanguage}{\mbox{$\sup {\rm C_{pc}}$}}
\newcommand{\supcnsublanguage}{\mbox{$\sup {\rm CN}$}}
\newcommand{\supcn}{\mbox{$\sup {\rm CN}$}}
\newcommand{\supmccn}{\mbox{$\sup {\rm mcCN}$}}
\newcommand{\suppc}{\mbox{$\sup {\rm PC}$}}
\newcommand{\suppn}{\mbox{$\sup {\rm PN}$}}
\newcommand{\supn}{\mbox{$\sup {\rm N}$}}
\newcommand{\TIME}{\mbox{$\rm TIME$}}
\newcommand{\transrel}{\mbox{$\rm Tr$}}
\newcommand{\trim}{\mbox{$\rm trim$}}
\newcommand{\trimgen}{\mbox{$\rm trimgen$}}
\newcommand{\triple}{\mbox{$\rm tri$}}
\newcommand{\uc}{\mbox{$\rm uc$}}

\newcommand{\exitset}{\mbox{$\rm ExitSet$}}

\newcommand{\cogstocsp}{\mbox{${\rm COGStocSP}$}}
\newcommand{\cvf}{\mbox{${\rm cvf}$}}
\newcommand{\gstocs}{\mbox{${\rm GStocS}$}}
\newcommand{\gstocsp}{\mbox{${\rm GStocSP}$}}
\newcommand{\gstoccsp}{\mbox{${\rm GStocCSP}$}}
\newcommand{\lqg}{\mbox{${\rm LQG}$}}
\newcommand{\mv}{\mbox{${\rm Min.Var}$}}
\newcommand{\ti}{\mbox{${\rm Time.Inv}$}}

\newcommand{\fstocs}{\mbox{${\rm FStocS}$}}
\newcommand{\fss}{\mbox{${\rm FSS}$}}
\newcommand{\fstocsp}{\mbox{${\rm FStocSP}$}}

\newcommand{\stoccs}{\mbox{${\rm StocCS}$}}
\newcommand{\stocs}{\mbox{${\rm StocS}$}}

\newcommand{\is}{\mbox{${\rm IS}$}}
\newcommand{\isdsrs}{\mbox{${\rm ISDSrs}$}}
\newcommand{\isdsones}{\mbox{${\rm ISDS1s}$}}
\newcommand{\isdsrscommon}{\mbox{${\rm ISDSrsCommon}$}}
\newcommand{\isdsonescommon}{\mbox{${\rm ISDS1sCommon}$}}
\newcommand{\isdsrsprivate}{\mbox{${\rm ISDSrsPrivate}$}}
\newcommand{\isnn}{\mbox{${\rm ISNN}$}}
\newcommand{\inn}{\mbox{${\rm I_{nn}}$}}

\newcommand{\bits}{\mbox{$\rm bits$}}

\newcommand{\dtime}{\mbox{$\rm DTIME$}}
\newcommand{\exptime}{\mbox{$\rm EXPTIME$}}
\newcommand{\np}{\mbox{$\rm NP$}}
\newcommand{\ntime}{\mbox{$\rm NTIME$}}
\newcommand{\polylog}{\mbox{$\rm polylog$}}
\newcommand{\timecomplexity}{\mbox{$\rm TIME$}}

\newcommand{\glnr}{\mbox{$Gl_{n}(\mathbb{R})$}}

\newcommand{\aut}{\mbox{${\bf Aut}$}}
\newcommand{\catset}{\mbox{${\bf Set}$}}
\newcommand{\comp}{\mbox{$comp$}}
\newcommand{\Grp}{\mbox{${\bf Grp}$}}
\newcommand{\kaut}{\mbox{${\bf K-Aut}$}}
\newcommand{\kmedv}{\mbox{${\bf K-Medv}$}}
\newcommand{\kmedvio}{\mbox{$({\bf K-Medv} \downarrow <I^+,O>)$}}
\newcommand{\moduler}{\mbox{${\bf Module_R}$}}
\newcommand{\ob}{\mbox{$ob$}}
\newcommand{\Set}{\mbox{${\bf Set}$}}
\newcommand{\cattopo}{\mbox{${\bf Topo}$}}

\newcommand{\cont}{\mbox{${\rm cont}$}}
\newcommand{\Deg}{\mbox{${\rm Deg}$}}
\newcommand{\degree}{\mbox{${\rm deg}$}}
\newcommand{\diff}{\mbox{${\rm diff}$}}
\newcommand{\gmon}{\mbox{${\rm G_{mon}}$}}
\newcommand{\diffrpos}{\mbox{${\rm Diff} \mathbb{R}_+$}}
\newcommand{\lcm}{\mbox{${\rm lcm}$}}
\newcommand{\mon}{\mbox{${\rm mon}$}}
\newcommand{\mnm}{\mbox{${\rm mnm}$}}
\newcommand{\order}{\mbox{${\rm order}$}}
\newcommand{\rpoly}[2]{\mbox{$R_+[#1]/(#1^{#2}-1)$}}
\newcommand{\slsp}{\mbox{${\rm SL}\Sigma{\rm P}$}}
\newcommand{\spoly}[2]{\mbox{$S_+[#1]/(#1^{#2}-1)$}}
\newcommand{\sqfree}{\mbox{${\rm sqfree}$}}
\newcommand{\support}{\mbox{${\rm support}$}}
\newcommand{\trdeg}{\mbox{${\rm trdeg}$}}

\newcommand{\as}{\mbox{{\rm $a.s.$}}} 
\newcommand{\aslim}{\mbox{{\rm $a.s.-\lim$}}} 
\newcommand{\ci}{\mbox{{\rm $CI$}}} 
\newcommand{\cig}{\mbox{{\rm $CIG$}}} 
\newcommand{\cigmin}{\mbox{{\rm $CIG_{min}$}}} 
\newcommand{\ciffg}{\mbox{$(F_1,F_2 | G ) \in \ci$}}
\newcommand{\cpdf}{\mbox{{\rm $CPDF$}}} 
\newcommand{\dlim}{\mbox{{\rm $D-\lim$}}} 
\newcommand{\essinf}{\mbox{{\rm $essinf$}}} 
\newcommand{\esssup}{\mbox{{\rm $esssup$}}} 
\newcommand{\foralmostall}{\mbox{{\rm $\mbox{almost all}$}}} 
\newcommand{\ift}{\mbox{{\rm $ift$}}} 
\newcommand{\ltwolim}{\mbox{{\rm $L_2-\lim$}}} 
\newcommand{\pdf}{\mbox{{\rm $PDF$}}} 
\newcommand{\pessinf}{\mbox{{\rm $P-essinf$}}} 
\newcommand{\pesssup}{\mbox{{\rm $P-esssup$}}} 
\newcommand{\plim}{\mbox{{\rm $P-\lim$}}} 

\newcommand{\aloc}{\mbox{{\rm ${\bf A_{loc}}$}}}
\newcommand{\alocplus}{\mbox{{\rm ${\bf A_{loc}^+}$}}}
\newcommand{\aone}{\mbox{{\rm ${\bf A_1}$}}}
\newcommand{\aplus}{\mbox{{\rm ${\bf A^+}$}}}
\newcommand{\bvar}{\mbox{{\rm $BV$}}}
\newcommand{\bvarc}{\mbox{{\rm $BV^c$}}}
\newcommand{\cadlag}{\mbox{{\rm c\`{a}dl\`{a}g}}}
\newcommand{\DL}{\mbox{{\rm $DL$}}}
\newcommand{\mone}{\mbox{{\rm $M_1$}}}
\newcommand{\monec}{\mbox{{\rm $M_1^c$}}}
\newcommand{\monepos}{\mbox{{\rm $M_{+,1}$}}}
\newcommand{\moneu}{\mbox{{\rm $M_{1u}$}}}
\newcommand{\moneuc}{\mbox{{\rm $M_{1u}^c$}}}
\newcommand{\moneuloc}{\mbox{{\rm $M_{1uloc}$}}}
\newcommand{\moneulocc}{\mbox{{\rm $M_{1uloc}^c$}}}
\newcommand{\mtwo}{\mbox{{\rm $M_2$}}}
\newcommand{\mtwos}{\mbox{{\rm $M_{2s}$}}}
\newcommand{\mtwosc}{\mbox{{\rm $M_{2s}^c$}}}
\newcommand{\mtwosd}{\mbox{{\rm $M_{2s}^d$}}}
\newcommand{\mtwosloc}{\mbox{{\rm $M_{2sloc}$}}}
\newcommand{\mtwoslocc}{\mbox{{\rm $M_{2sloc}^c$}}}
\newcommand{\mtwoslocd}{\mbox{{\rm $M_{2sloc}^d$}}}
\newcommand{\mtwoc}{\mbox{{\rm $M_2^c$}}}
\newcommand{\var}{\mbox{{\rm $Var$}}}
\newcommand{\semim}{\mbox{{\rm $SemM$}}}
\newcommand{\semimc}{\mbox{{\rm $SemM^c$}}}
\newcommand{\semimone}{\mbox{{\rm $SemM_1$}}}
\newcommand{\semimonec}{\mbox{{\rm $SemM_1^c$}}}
\newcommand{\semimspecial}{\mbox{{\rm $SemM_s$}}}
\newcommand{\semimtwo}{\mbox{{\rm $SemM_2$}}}
\newcommand{\semimtwoc}{\mbox{{\rm $SemM_2^c$}}}
\newcommand{\semimloc}{\mbox{{\rm $SemM_{loc}$}}}
\newcommand{\semimlocc}{\mbox{{\rm $SemM_{loc}^c$}}}
\newcommand{\stoppingtimes}{\mbox{{\rm $T_{st}$}}}
\newcommand{\stoppingtimesinfty}{\mbox{{\rm $T_{st \uparrow \infty}$}}}
\newcommand{\submone}{\mbox{{\rm $SubM_1$}}}
\newcommand{\submonec}{\mbox{{\rm $SubM_1^c$}}}
\newcommand{\submonepos}{\mbox{{\rm $SubM_{+,1}$}}}
\newcommand{\submpos}{\mbox{{\rm $SubM_+$}}}
\newcommand{\submposc}{\mbox{{\rm $SubM_+^c$}}}
\newcommand{\supmone}{\mbox{{\rm $SupM_1$}}}

\newcommand{\deficiency}{\mbox{{\rm dfc}}}
\newcommand{\reactionnet}{\mbox{{\rm rnet}}}
\newcommand{\rnet}{\mbox{{\rm $rnet$}}}
\newcommand{\netc}{\mbox{{\rm $net_c$}}}
\newcommand{\nets}{\mbox{{\rm $net_s$}}}

\newcommand{\lane}{\mbox{$\rm lane$}}
\newcommand{\od}{\mbox{$\rm OD$}}
\newcommand{\pldim}{\mbox{$\rm m$} \times \mbox{$\rm km/h$} \times \mbox{$\rm veh}}
\newcommand{\roadnet}{\mbox{$\rm RoadNet$}}
\newcommand{\roadsection}{\mbox{$\rm Secs$}}
\newcommand{\subnet}{\mbox{$\rm SubNet$}}
\newcommand{\subnetin}{\mbox{$\rm SubNet_{in}$}}
\newcommand{\subnetout}{\mbox{$\rm SubNet_{out}$}}
\newcommand{\subnetlink}{\mbox{$\rm R_{link}$}}

\newcommand{\dl}{D_{\lambda}}
\newcommand{\occ}{\mbox{\rm Occ}}

\newcommand{\ta}{\widetilde{\alpha}}
\newcommand{\tb}{\widetilde{\beta}}
\newcommand{\tg}{\widetilde{\gamma}}
\newcommand{\td}{\widetilde{\delta}}
\newcommand{\te}{\widetilde{\varepsilon}}
\newcommand{\tx}{\widetilde{\xi}}
\newcommand{\tp}{\widetilde{p}}
\newcommand{\tm}{\widetilde{M}}
\newcommand{\wt}[1]{\widetilde{#1}}
\newcommand{\wh}[1]{\widehat{#1}}

\newcommand{\TT}{T\!\!\!\! I}
\newcommand{\DD}{D\!\!\!\! I}
\newcommand{\RR}{R\!\!\!\! I}
\newcommand{\CC}{C\!\!\!\! I}
\newcommand{\NN}{N\!\!\!\!\!\! I\;}      
 

\maketitle
\thispagestyle{empty}
\pagestyle{empty}

\begin{abstract}
An observable canonical form is formulated
for the set of rational systems on a variety
each of which is a single-input-single-output,
affine in the input, and
a minimal realization of its response map.
The equivalence relation for the canonical form
is defined by the condition
that two equivalent systems have the same response map.
A proof is provided that the defined form is well-defined canonical form.
Special cases are discussed.
\end{abstract}
%
\section{INTRODUCTION}\label{sec:intro}
%
%
The purpose of the paper is to define
an observable canonical form for a rational system on a variety
which is single-input-single-output and affine in the input.
\par
The motivation of a canonical form for this set of systems
is their use in system identification.
In case of blackbox modeling with a rational system, 
a canonical form is needed.
If no canonical form is used then there arises
a problem of nonidentifiability of the system parametrization
with serious consequences.
The canonical form in this case is based on the equivalence relation 
which relates two rational systems if they have the same response map.
\par
There is another motivation which is control synthesis of rational systems.
For this a canonical form is needed
for both response-map equivalence and for feedback equivalence.
The research issue to define a canonical form
of rational systems for these combined equivalence relations
is briefly discussed in the paper but it is not the main focus.
\par
The main results of the paper are the formulation of the concept of
the observable canonical form of a rational system and
the theorem that the observable canonical form is well defined.
\par
A summary of the remainder of the paper follows.
The next section provides a more detailed problem formulation.
Rational systems are defined in Section \ref{sec:rationalsystems}
while the canonical form is defined in Section \ref{sec:canonicalforms}.
That the defined observable canonical form is a well-defined canonical
form is established in Section \ref{sec:theory}.
%
\section{PROBLEM FORMULATION}\label{sec:problem}
To be able to motivate the use of canonical forms,
it has to be explained what they are.
\par
Consider a set of control systems,
each system of which has an input, a state, and an output.
Associate with each system its response map
which maps, for every time,
the past input trajectory to the output at that time.
The realization problem is the converse issue.
Consider an arbitrary response map,
which for every time maps a past input trajectory to the output at that time.
The realization problem is to construct, for a considered response map,
a system in a considered set of which the associated response map
equals the considered response map.
A condition must hold for an arbitrary response map
to have a realization as a system in the considered set.
But if the condition holds,
then there may exist not one but many systems having the considered response map.
Hence one restricts attention to those systems representing the response map
which are minimal in a sense to be defined,
usually related to the dimension of the state set.
Again, there is not a unique minimal system representing
the considered response map but a set of systems.
Consider thus the subset of systems 
each of which is a minimal realization of its own response map.
Define then an equivalence relation on any tuple of that subset of systems
if both systems have the same response map.
\par
A canonical form or normal form is now a subset of the subset
of the considered systems for the above defined equivalence relation.
A canonical form requires that each system in the considered subset is equivalent
to a unique element of the canonical form.
\par
A major motivation for canonical forms is system identification.
In system identification,
a canonical form of a set of systems can be used
to restrict the problem of how to estimate the parameters of the system.
On the set of minimal systems
one defines the equivalence relation of two systems having the same response map.
Without a canonical form there is an identifiability problem
meaning that there exist two or more different minimal systems
which represent the same response map.
\par
A second application of a canonical form is control synthesis.
In this case the equivalence relation is based on
(1) the equality of the response map, and
(2) feedback equivalence.
Feedback equivalence is defined for two systems
if the second system can be obtained from the first system
by a state-feedback control law and a new input.
A canonical form for these properties simplifies the control synthesis.
\par
Canonical forms for time-invariant finite-dimensional linear systems
have been formulated and proven.
P. Brunovsky formulated a canonical form, \cite{brunovsky:1970}.
W.M. Wonham and A. Morse, \cite{wonham:morse:1972},
have formulated a canonical form for the equivalence
of the response map and feedback equivalence.
The well known text book, \cite{wonham:1979}, also describes this well.
Another early paper is \cite{popov:1972}.
Books which contain a discussion on canonical forms for linear systems include:
\cite[Section 9.2]{callier:desoer:1991};
\cite[pp. 187-192, pp. 198-199, Section 5.5, Section 6.3, Section 6.4]{chen:1984};
\cite[Sections 6.7.3, 7.1]{kailath:1980:book};
\cite[pp. 494-508]{padulo:arbib:1974};
\cite[p. 292]{sontag:1998:book}; and
\cite[p. 39, Section 5.5]{wonham:1974}.
\par
In the book
\cite[pp. 137--142]{isidori:2001}
there is defined a normal form in local coordinates
for a single-input-single-output
and affine-in-the-input smooth nonlinear system on the state set $\mathbb{R}^n$.
But that form is not a canonical form as defined in this paper.
Note that the normal form of that reference is obviously not controllable.
There is neither a proof nor a claim in that reference
that the normal form is a well-defined canonical form.
\par
I.A. Tall and W. Respondek,
\cite{tall:respondek:2003},
have formulated a canonical form
for smooth nonlinear systems on a differential manifold
for the equivalence relation of feedback equivalence.
The approach of this paper differs from that of the paper of Tall and Respondek
in that they consider feedback equivalence
while in this paper only the equivalence of the response map is addressed
and that in this paper
the variety plays a role in the formulation of the canonical form.
A rational system on a variety is characterized
by both the variety and by the system.
This makes the problem more complicated than the contribution
of \cite{tall:respondek:2003}.
\begin{problem}\label{problem:rsystemvarietycanonicalform}
	Consider the set of rational systems on a variety,
	each of which is single-input-single-output, 
	affine in the input, and
	each of which is a minimal realization of its response map.
	Formulate a candidate canonical form for this set of systems
	and prove that it is a well-defined canonical form.
\end{problem}
The extensions to multi-input-multi-output and to other subsets of rational systems
require much more space then is available in this short paper.
%
\section{RATIONAL SYSTEMS}\label{sec:rationalsystems}
\par
Terminology and notation of commutative algebra
is used from the books
\cite{eisenbud:1995,becker:weisspfenning:1993,bochnak:coste:roy:1998,zariski:samuel:1958,zariski:samuel:1960}.
References on algebraic geometry include,
\cite{cox:little:oshea:1992,harris:1992,hartshorne:1977,shafarevich:1994,shafarevich:1994b}.
\par
The notation of the paper is simple.
The set of the {\em integers} is denoted by $\mathbb{Z}$
and the set of the {\em strictly positive integers} by
$\mathbb{Z}_+ = \{ 1,2, \ldots \}$.
For $n \in \mathbb{Z}_+$ define
$\mathbb{Z}_n = \{ 1, 2, \ldots, n \}$.
The set of the {\em natural numbers} is denoted by
$\mathbb{N} = \{ 0, 1, 2, \ldots \}$
and, for $n \in \mathbb{Z}_+$, $\mathbb{N}_n = \{ 0, 1, 2, \ldots, n \}$.
The set of the real numbers is denoted by $\mathbb{R}$ and
that of the positive and the strictly-positive real numbers
respectively by
$\mathbb{R}_+ = [0, \infty)$ and $\mathbb{R}_{s+} = (0,\infty)$.
The {\em vector space of $n$-tuples of the real numbers} is denoted by
$\mathbb{R}^n$.
\par
A subset $X \subseteq \mathbb{R}^n$ for $n \in \mathbb{Z}_+$
is called a {\em variety}
if it is determined by a finite set of polynomial equalities.
Such a set is also called an {\em algebraic set}.
\par
A variety is called {\em irreducible} if it cannot be written
as a union of two disjoint nonempty varieties.
Any variety determined by a set of polynomials $p_1, ~ \dots, p_k$
of the form,
\begin{eqnarray*}
	X
	& = & \{ x \in \mathbb{R}^n | ~ 0 = p_i(x), ~ \forall ~ i \in \mathbb{Z}_k \}, ~
	      n, ~ k \in \mathbb{Z}_+,
\end{eqnarray*}
is an irreducible variety according to
\cite[Section 5.5, Proposition 5]{cox:little:oshea:1992}.
A canonical form for an irreducible variety can be formulated
based on the concept of a decomposition of such a variety using prime ideals.
This will not be detailed in this short paper.
\par
Denote the {\em algebra of polynomials} in $n \in \mathbb{Z}_+$
{\em variables} with real coefficients by
$\mathbb{R}[X_1, \ldots, X_n]$.
For a variety $X$,
denote by $I(X)$ the ideal of polynomials of
$\mathbb{R}[X_1, \ldots, X_n]$ which vanish on the variety $X$.
The elements of
$\mathbb{R}[X_1, \ldots, X_n] / I(X)$
are referred to as
{\em polynomials on the variety} $X$.
The ring of all such polynomials is denoted by $A_X$
which is also an algebra.
If the variety $X$ is irreducible
then the ring $A_X$ is an integral domain
hence one can define
the {\em field of rational functions on the variety} $X$
as a field of fractions $Q_{X}$ of the algebra $A_X$.
\par
A polynomial and a rational function for $n$ variables
are denoted respectively by the representations
(assumed to be defined over a finite sum),
\begin{eqnarray*}
	    p(x)
	& = & \sum_{k \in \mathbb{N}^n} c_p(k) \prod_{i=1}^n x_i^{k_i}
	      = \sum_{k \in \mathbb{N}^n} c_p(k) x^k ~\\
        &   & \in \mathbb{R}[X_1, \ldots, X_n], ~
	      (\forall ~ k \in \mathbb{N}^n, ~ c_p(k) \in \mathbb{R}), \\
	    r(x)
	& = & \frac{p(x)}{q(x)}.
\end{eqnarray*}
For a rational function the following special form or canonical form is defined:
(1) there are no common factors in the numerator and the denominator;
such factors when present can be eliminated by cancellation; and
(2) the constant factor in the denominator polynomial, assumed to be present,
is set to one by multiplication of the numerator and the denominator
with a real valued number.
In case there is no constant term in the denominator then
the coefficient of the highest degree in a defined ordering
of the denominator polynomial, 
has to be set to the real number one.
\begin{eqnarray*}
	    r(x)
	& = & \frac{p(x)}{q(x)}, \\
	    q(x)
	& = & 1 + \sum_{k \in \mathbb{N}^n \backslash \{0\}} c_q(k) ~ \prod_{i=1}^n x_i^{k(i)}, \\
	    Q_{X,can}
	& = & \{ r(x) \in Q_X, ~ \mbox{as defined above} \}.
\end{eqnarray*}
\par
The {\em transcendence degree} of a field $F$,
denoted by $\trdeg (F)$,
is defined to be
the greatest number of al\-ge\-bra\-i\-cal\-ly-independent elements 
of $F$ over $\mathbb{R}$,
\cite[Section 7.1, p. 293, p. 304]{becker:weisspfenning:1993} and
\cite[Ch. 2, Sections 3 and 12]{zariski:samuel:1958}.
\par
For the detailed definitions of the concepts introduced below,
a rational system on a variety etc.,
the reader is referred to the papers,
\cite{nemcova:schuppen:2009,nemcova:schuppen:2010:aam,nemcova:petreczky:schuppen:2016:cdc:red}.
This includes the concept of a differential equation on a variety,
see also \cite{yuanwang:sontag:1992:siamjco}, 
and the fact that
the variety is forward invariant with respect to the differential equation.
Controlled invariant hypersurfaces of polynomial control systems on varieties 
are considered in
\cite{zerz:walcher:2012,schilli:zerz:levandovskyy:2014}.
\begin{definition}\label{def:rationalsystem}
	A {\em rational system} on a variety,
	in particular,
	a single-input-single-output-system which is affine in the input,
	is defined as a control system as understood in control theory
	with the representation,
	\begin{eqnarray}
		        dx(t)/dt
		    & = & f_0(x(t)) + f_1(x(t)) u(t), ~ x(0) = x_0, \label{eq:rsystemdx}\\
		        y(t)
		    & = & h(x(t)), \label{eq:rsystemy} \\
		    f_{\alpha}
		    & = & \sum_{i=1}^n [ f_{0,i}(x) + f_{1,i}(x) \alpha ]
		                       \frac{\partial}{\partial x_i}, ~
				       \forall ~ \alpha \in U, \\
		    f
		& = & \{ f_{\alpha}, ~ \alpha \in U\}, \\
		    s
		& = & (X, U, Y, f_0, f_1, h, x_0) \in S_r;
	\end{eqnarray}
	where $n \in \mathbb{Z}_+$,
	$X \subseteq \mathbb{R}^n$ is
	an irreducible nonempty variety called the {\em state set},
	$U \subseteq \mathbb{R}$ is called the {\em input set},
        it is assumed that $\{0\} \subseteq U$ and 
	that $U$ contains at least two distinct elements,
	$Y = \mathbb{R}$ is called the {\em output set},
	$ x_0 \in X$ is called the {\em initial state},
	$f_{0,1}, ~ \ldots, f_{0,n}, ~ f_{1,1}, \ldots f_{1,n} \in Q_X$ and
	$h \in Q_x$ are rational functions on the variety,
	$u: [0,\infty) \rightarrow U$ is a piecewise-constant input function, and
	$S_r$ denotes the set of rational systems as defined here.
        \par	
        One defines the set of piecewise-constant input functions
        which are further restricted by the existence of a solution
        of the rational differential equation of the system.
        Denote for a rational system $s \in S_r$ as defined above,
        the admissible set of piecewise-constant input functions by $U_{pc}(s)$.
        Further, for $u \in U_{pc}(s)$,
        $t_u \in \mathbb{R}_+$ denotes the life time of the solution
	of the differential equation for $x$ with input $u$.
	For any $u \in U_{pc}(s)$ and any $t \in [0,t_u)$
	denote by $u[0,t)$
	the restriction of the function $u$ to the interval $[0,t)$.
	Denote the solution of the differential equation (\ref{eq:rsystemdx}) by
	$x(t; x_0, u[0,t))$, $\forall ~ t \in [0,t_u)$.
        \par
	Define the {\em dimension} of this rational system
	as the transcendence degree of the set of rational functions on $X$,
	$\dim(X) = \trdeg(Q_X)$.
	In the remainder of the paper, a rational system will refer
	to a rational system on a variety.
\end{definition}
\begin{definition}\label{def:responsemap}
	Associate with any rational system in the considered set,
	its {\em response map} as,
	\begin{eqnarray*}
		&   & r_s: U_{pc}(s) \rightarrow \mathbb{R}, ~
		      r_s \in A(U_{pc} \rightarrow \mathbb{R}), ~\\
                &   & \mbox{such that for} ~ e = \mbox{empty input,} ~ r_s(e) = 0,
	\end{eqnarray*}
        and such that for all $u \in U_{pc}(s)$,
	if $x: [0,t_u) \rightarrow X$
	is a solution of the rational system $s$ for $u$,
	i.e. $x$ satisfies the differential equation
	(\ref{eq:rsystemdx}) and the output (\ref{eq:rsystemy}) is well defined,
	then
	\begin{eqnarray*}
		    r_s(u[0,t))
		& = & y(t) = h(x(t; x_0, u[0,t))), ~ \forall ~ t \in [0,t_u).
	\end{eqnarray*}
\end{definition}
The realization problem for the considered set of rational systems is,
when considering an arbitrary response map,
to determine whether there exists a rational systems
whose response map equals the considered response map,
\cite{nemcova:schuppen:2009,nemcova:schuppen:2010:aam}.
\begin{definition}\label{def:realization}
	Consider a response map,
	$p: U_{pc} \rightarrow \mathbb{R}$.
        Call the system $s \in S_r$
	a {\em realization} of the considered map $p$ if
	\begin{eqnarray*}
		&   & \forall ~ u \in U_{pc}, ~
		      \forall ~ t \in [0,t_u), ~
		      p(u[0,t)) = r_s(u[0,t)).
	\end{eqnarray*}
	Call the system a {\em minimal realization} of the response map
	if $\dim (X) = \trdeg (Q_{obs}(p))$.
	Define a {\em minimal rational system}
	to be a rational system which is a minimal realization
	of its own response map.
\end{definition}
In general, a realization is not unique.
Attention will be restricted to minimal realizations
characterized by a condition of controllability and of observability
defined next.
Observability of discrete-time polynomial systems was defined in
\cite{sontag:1979:phdthesis},
observability of continuous-time polynomial systems was defined in
\cite{bartosiewicz:1988}, and
rational observability and algebraic controllability of
rational systems were defined in
\cite{nemcova:schuppen:2010:aam}.
The formal definitions are recalled for ease of reference.
\begin{definition}\label{def:rationalobservability}
	Consider a rational system as defined in Def. \ref{def:rationalsystem}.
	The {\em observation algebra} $A_{obs}(s) \subseteq Q_X$
	of a rational system $s \in S_r$
	is defined as the smallest subalgebra of $Q_X$
	which contains the $\mathbb{R}^1$-valued output map $h$ and
	is closed with respect to taking Lie derivatives
	of the vector field of the system.
	\par
	Denote by $Q_{obs}(s) \subseteq Q_X$
	the field of fractions of $A_{obs}(s)$
	and call this set the {\em observation field} of the system.
	Call the system $s \in S_r$ {\em rationally observable} if
	\begin{eqnarray}
		    Q_X
		& = & Q_{obs}(s).
	\end{eqnarray}
\end{definition}
\begin{definition}\label{def:algebraiccontrollability}
	Consider a rational system as defined in Def. \ref{def:rationalsystem}.
	Call the system {\em algebraically controllable}
	or {\em algebraically reachable} if
	\begin{eqnarray*}
		X
		& = & \zclosure (\{x(t_u) \in X | ~ u \in U_{pc} \}).
	\end{eqnarray*}
	where $\zclosure(S)$ of a set $S \subset \mathbb{R}^n$
	denotes the smallest variety containing the set $S$,
	also called the {\em Zariski-closure} of $S$.
\end{definition}
See \cite{nemcova:schuppen:2017:ifacwc}
for procedures how to check that algebraic controllability holds.
\par
It follows from the existing realization theory
that if a rational system is algebraically controllable
and rationally observable
then it is a minimal realization of its response map,
\cite[Proposition 6]{nemcova:schuppen:2010:aam}.
\par
A minimal realization of a response map is not unique.
It has been proven that
any two minimal rational realizations are birationally equivalent
if the condition holds that the elements of $Q_X \backslash Q_{obs}(s)$
are not algebraic over $Q_{obs}(s)$ for both systems $s$.
Let $X\subseteq \mathbb{R}^n$ and $X'\subseteq \mathbb{R}^{n'}$
be two irreducible varieties.
A {\em birational map} from $X$ to $X'$
is a map which has $n'$ components which are all rational functions of $Q_X$ and
for which an inverse exists such that
it is a map which has $n$ components which are all rational functions of $Q_{X'}$.
A birational map transforms a rational system on a variety
to another rational system on another variety, see \cite{nemcova:schuppen:2010:aam}.
A reference on birational geometry is
\cite{kollar:mori:1998}.
%
\section{CANONICAL FORMS}\label{sec:canonicalforms}
A canonical form is defined
for the case one has a set with an equivalence relation defined on it.
The terms of normal form, canonical form, or canonical normal form,
all refer to the same concept.
The authors prefer the expression {\em canonical form}.
The following books define the concept of canonical form,
\cite[p. 277]{birkhoff:maclane:1977},
\cite[Section 0.3]{jacobson:1985},
\cite[Subsection 2.2.1]{wechler:1992}, and
\cite[Section 4.5 Reduction Relations]{becker:weisspfenning:1993}.
\begin{definition}\label{def:canonicalform}
Consider a set $X$ and
an equivalence relation $E \subseteq X \times X$ defined on it.
A {\em canonical form} or a {\em normal form}
for this set and this equivalence relation,
consists of a subset $X_c \subseteq X$
such that, for any $x \in X$,
there exists a unique element of the canonical form $x_c \in X_c$,
such that $(x,x_c) \in E$.
\end{definition}
A canonical form is nonunique in general.
One may impose conditions on the canonical form
if there is an algebraic structure on the underlying set.
\par
In system theory,
a canonical form is needed for the realization of a response map.
Realization theory provides a condition
for the existence of a realization of a response map.
One then restricts attention to the subset of minimal systems;
equivalently, those systems which are minimal realizations of their own response map.
One then defines an equivalence relation on the set of minimal systems
if the response maps of the two considered miminal systems are equal.
\par
In system theory there have been defined canonical forms
for two different equivalence relations defined next:
\begin{enumerate}
\item
response-map equivalence; or
\item
feedback-and-response-map equivalence.
\end{enumerate}
\begin{definition}\label{def:equivalenceresponsemap}
	Consider the set $S_{r,\min}$ of minimal rational systems on a variety $X$.
	Define the {\em response-map equivalence relation} $E_{rm}$
	on this set of systems by the condition
	that two systems are equivalent if their response maps are equal.
	One then says that the considered two systems
	are {\em response-map equivalent}.
\end{definition}
\begin{definition}
	Consider the set $S_{r,\min}$ of minimal rational systems on a variety $X$.
	Define the {\em feedback-and-response-map equivalence relation} $E_{frm}$
	on this set of systems by the condition
	that two systems are response-map equivalent and state-feedback equivalent.
	System 1 and System 2 are called {\em state-feedback equivalent}
	if System 1 is
	response-map equivalent with System 2 after
	closing the loop with a state-feedback and a new input variable $v$ according to,
	\begin{eqnarray}
		    u
		& = & g_0(x) + g_1(x) v, ~ g_0, ~ g_1 \in Q_X, \\
		    dx(t)/dt
		& = & [ f_0(x(t)) + f_1(x(t)) g_0(x(t)) ] + \nonumber \\
		&   & + f_1(x(t)) g_1(x(t)) v(t),\\
		    y(t)
		& = & h(x(t),
	\end{eqnarray}
	One then says that the considered two systems
	are {\em feedback-and-response-map equivalent}.
\end{definition}
Feedback equivalence is best considered in combination
with the control canonical form.
For an observable canonical form
one may define an observer-feedback equivalence relation
not further discussed in this paper.
\begin{problem}\label{problem:rsystemcanonicalform2}
	Consider the set $S_{r,\min}$ of rational systems
	each of which is single-input-single-output and affine in the input, and
	and each of which is a minimal realization of its response map.
	The problem is to define a {\em canonical form}
	for (1) the response-map equivalence and
	for (2) the feedback-and-response-map equivalence.
	Prove that each of the defined forms is a well-defined canonical form.
\end{problem}
In this paper only a canonical form for the first equivalence relation
is provided.
The solution for the second equivalence relation is postponed.
\par
For a set of systems and
for the response-map equivalence relation
there is no unique canonical form.
For time-invariant finite-dimensional linear systems
there have been defined
both a control canonical form and an observable canonical form.
Which of the many canonical form is most appropriate depends on other objectives
of the user.
\par
The formulation of a canonical form for the set of minimal rational systems
as specified above
has to involve:
(1) the variety of the state set $X$ and
(2) the functions specifying the minimal rational system,
both the dynamics and the output map.
This makes the problem different from that of
a canonical form for a time-invariant linear system
which are defined on the state-space $\mathbb{R}^n$
and by the functions of the dynamics and the output equation.
In the linear case there is no restriction on the state space except its dimension.
In the book
\cite[Sections 27 and 28]{sontag:1979:phdthesis}
there is a discussion of these two cases
for discrete-time polynomial systems.
\par
Due to the above remark, there are two types of canonical forms for rational systems:
\begin{enumerate}
	\item
		with a structured rational systems and an unrestricted arbitrary variety; and
	\item
		with a variety of a given structure and
		with un unstructured or partly-structured rational system.
\end{enumerate}
\par
One may also define a canonical form for the description
of the variety of the state set.
The classification of algebraic varieties up to a birational equivalence
is the main problem studied within birational geometry,
see \cite{kollar:mori:1998}.
It is proven that every $n$-dimensional irreducible variety
over an algebraically closed field
is birationally equivalent to a hypersurface in $\mathbb{R}^{n+1}$.
Hence, good candidates for
a canonical form for the description of the variety of the state set
are hypersurfaces in $\mathbb{R}^{n+1}$.
Note that a hypersurface is given by a homogeneous polynomial,
a polynomial whose nonzero terms all have the same degree.
\begin{definition}\label{def:rsystemobscanformsiso}
Consider the set of rational systems as defined in Def. \ref{def:rationalsystem}.
Assume that the class of systems is restricted from $S_r$ to $S_{rr}$ 
so that for any system $s$ in the considered class $S_{rr}$,
$Q_X \backslash Q_{obs}(s)$ is not algebraic over $Q_{obs}(s)$.
\par
Define the {\em observable canonical form}
on the set of minimal rational systems $S_{rr}$
for the response-map equivalence relation
as the algebraic structure described by the equations,
\begin{eqnarray}
      X
  & = & \{ x \in \mathbb{R}^n | 0 = p_i(x), ~ \forall ~ i \in \mathbb{Z}_k \},
         \nonumber\\
  &   & \mbox{an irreducible nonempty variety,} \nonumber \\
      d
  & = & \trdeg (Q_{obs}(s)) \in \mathbb{Z}_+, ~
        n, ~ k \in \mathbb{Z}_+, \nonumber \\
      dx_1(t)/dt
  & = & x_2(t) + f_{1}(x(t)) ~ u(t), 
\end{eqnarray}
\begin{eqnarray}
      dx_i(t)/dt
  & = & x_{i+1}(t) + f_{i}(x(t)) ~ u(t), ~  \\
  &   & i = 2, 3, \ldots, n-1, \nonumber \\
      dx_n(t)/dt
  & = & f_{n,0}(x(t)) + f_{n,1}(x(t)) ~ u(t), \\
      y(t)
  & = & x_1(t), ~~  \\
	&   & f_{1}, \ldots, f_{n-1}, ~ f_{n,0}, ~ f_{n,1} \in Q_{X, can}, \nonumber \\
  &   & (\forall ~ x \in X \backslash A_e, ~
	   f_{n,1}(x) \neq 0); \nonumber  \\
  &   & \mbox{where} ~ A_e \subset X ~ \mbox{is an algebraic set,}
	  \nonumber  \\
      S_{rr,ocf}
	& \subset & S_{rr}, ~
	\mbox{denotes the subset of rational systems} \nonumber \\
  &   & \mbox{in the observable canonical form.} \nonumber
\end{eqnarray}
Every system of the observable canonical form
is assumed to be algebraically controllable.
\end{definition}
\par
The assumption algebraic controllability in the above definition
of an observable canonical form is necessary due to the focus on observability.
Even in the case of the obervable canonical form of a time-invariant minimal linear system,
the condition of controllability has to be imposed.
It is conjectured that the condition that
$(\forall ~ x \in X \backslash A_e, ~ f_{n,1}(x) \neq 0)$
is a necessary condition for algebraic controllability.
The reader finds in the paper \cite{nemcova:schuppen:2017:ifacwc}
several ways to calculate whether or not a rational system
is algebraically controllable.
\par
The authors are aware of the control canonical form of nonlinear systems
on manifolds of \cite{tall:respondek:2003}.
In that canonical form,
one extracts from the nonlinear functions
linear terms in the state-input-pair and
a quadratic term in the state multiplying
a series representation of rational functions of homogeneous degrees.
The authors of this paper have decided not to use that particular
canonical form in this paper.
\par
To show that the above defined observable canonical form
is a well-defined canonical form
it has to be proven that:
\begin{enumerate}
	\item
		every rational system in the observable canonical form
		is a minimal realization of its response map;
	\item
		for every rational system which is a minimal realization
		of its response map,
		there exists a rational system in the observable canonical form
		such that the two systems are response-map equivalent;
	\item
		if two rational systems in the observable canonical form
		are response-map equivalent
		then they are identical.
\end{enumerate}
The proofs of the above items are provided in Section \ref{sec:theory}.
\par
There follow the special cases
of rational systems in the observable canonical form
for systems with the state-space dimensions one and two.
\begin{example}\label{ex:rsystemocf1n1}
	Consider the rational system
	with state-space dimension one, ($n=1$).
	\begin{eqnarray}
		    X
		& = & \{ x \in \mathbb{R}^1 | 0 = p_1(x) = \ldots = p_k(x) \}, ~
		      k \in \mathbb{Z}_+, \nonumber \\
		    dx(t)/dt
		& = & f_0(x(t)) + f_1(x(t)) ~ u(t), ~ x(0) = x_0, \\
		    y(t)
		& = & x(t), ~ \\
		&   & f_0, ~ f_1 \in Q_{X,can}, ~
		      (\forall ~ x \in X \backslash A_e, ~
		      f_1(x) \neq 0), ~ \nonumber \\
		&   & A_e \subset A_X ~ \mbox{an algebraic set,} \nonumber
	\end{eqnarray}
	Assume in addition that the system is algebraically controllable.
        Then this particular system is in the observable canonical form.
	The state set of this system is an irreducible variety in $\mathbb{R}^1$.
	An irreducible variety in $\mathbb{R}^1$ is either a singleton or all of $\mathbb{R}^1$.
\end{example}
%
%
\begin{example}\label{ex:rsystemocf1n2}
	Consider the rational system with state-space dimension two, $n=2$.
	\begin{eqnarray}
		    X
		& = & \{ x \in \mathbb{R}^2 | 0 = p_1(x) = \ldots p_k(x) \}, ~
		      k \in \mathbb{Z}_+, \nonumber  \\
		    \frac{dx_1(t)}{dt}
		& = & x_2(t) + f_{1}(x_1(t)) ~ u(t), \\
		    \frac{dx_2(t)}{dt}
		& = & f_{2,0}(x(t)) + f_{2,1}(x(t)) ~ u(t), ~  \\
		    y(t)
		& = & x_1(t), ~ \\
		&   & f_1, ~ f_{2,0}, ~ f_{2,1} \in Q_{X, can}, \nonumber \\
		&   & (\forall (x_1, x_2) \in X \backslash A_e,  ~
		      f_{2,1}((x_1,x_2)) \neq 0
                      ). \nonumber
	\end{eqnarray}
	Assume in addition that the system is algebraically controllable.
        Then this system is in the observable canonical form.
\end{example}
\begin{example}\label{ex:rsystemocfn2specific}
	Consider the specific rational system
	with state-space dimension two, $n=2$.
	\begin{eqnarray}
		    X
		& = & \{ x \in \mathbb{R}^2 | 0 = p_1(x) = \ldots = p_k(x) \}, ~
		      k \in \mathbb{Z}_+, \nonumber \\
		    \frac{dx_1(t)}{dt}
		& = & x_2(t), \\
		    \frac{dx_2(t)}{dt}
		& = & - \frac{x_2(t)}{1 + x_1(t)^2} ~ u(t), ~ \\
		    y(t)
		& = & x_1(t).
	\end{eqnarray}
	Assume in addition that the system is algebraically controllable.
        Then this particular system is in the observable canonical form.
\end{example}
\begin{definition}\label{def:observabilityindex}
	Consider a minimal rational system.
	Define the
	{\em observability index with respect to response-map equivalence}
	of this system, as the integer,
	\begin{eqnarray}
		n_o
		& = & \min_{k \in \mathbb{Z}_+, ~ Q_{obs}(s) = Q(G_k)} ~ |G_k|, \\
		    G_k
		& = & \left\{
	              \begin{array}{l}
			      h, L_{f_{\alpha_1}} h, ~ L_{f_{\alpha_2}} L_{f_{\alpha_1}} h,
			      \ldots \\
			      L_{f_{\alpha_k}} L_{f_{\alpha_{k-1}}} \ldots L_{f_{\alpha_1}} h, \\
			      \forall ~ \alpha_1, \ldots, \alpha_k \in U
	              \end{array}
		      \right\}, \\
		    L_{f_{\alpha}}
		& = & \sum_{i=1}^n [ f_0(x) + f_1(x) \alpha ]
		                   \frac{\partial}{\partial x_i}.
	\end{eqnarray}
	In words, the observability index is the minimal number
	of elements of the set $G_k$
	consisting of the output map $h$ and its Lie derivatives upto order $k-1$
	such that this set is a generator set of the observation field
	$Q_{obs}(s)$ of the system.
	Here $|G_k|$ denotes the number of elements in the set $G_k$.
\end{definition}
The dimension of an irreducible variety $X$ is defined as its Krull dimension and 
coincides with the transcendence degree of its function field $Q_X$. 
Hence, the dimension of the state set $X \subseteq \mathbb{R}^n$ 
of a rational system $s$ as well as the transcendence basis of $Q_X $ 
is finite ($\leq n$). 
Recall that the transcendence degree is defined as the smallest number
of generators of a field which are algebraically independent.
If the system $s$ is rationally observable, 
such that $Q_X = Q_{obs}(s)$, 
the generators of $Q_X$ can be chosen from the set 
$\{h, L_{f_{\alpha_1}} h, \ldots, L_{f_{\alpha_k}} \ldots L_{f_{\alpha_1}} h ~|~ 
\forall ~ \alpha_1, \ldots, \alpha_k \in U, k \in \mathbb{N}\}$. 
Moreover, 
the finiteness of a generator set of $Q_X$ 
implies that only finitely many Lie derivatives from $Q_{obs}(s)$ 
are sufficient to derive the generators of $Q_X$. 
Hence, the observability index $n_o$ defined in 
Definition \ref{def:observabilityindex} is a finite integer. 
Depending on the dimension of the output set and 
on the structure of $h$ and $f_{\alpha}$'s, 
the value of $n_o$ can be smaller, equal, or greater than $n$.
Below an example will be presented where $n < n_o$.
Observability indices can also be defined
for multi-input multi-output rational systems,
along the lines of
\cite{popov:1972,wonham:morse:1972} for linear systems.
%

\begin{example}\label{ex:obsindexhigherthann}
	Consider the polynomial system $s$,
	\begin{eqnarray*}
		X
		& = & \mathbb{R}^2, \\
		    dx(t)/dt
		& = & \left(
		      \begin{array}{l}
			      x_1(t)-x_2^2(t)+x_2(t) \\
			      x_2(t)
		      \end{array}
		      \right), ~ x(0) = x_0, \\
		    y(t)
		& = & x_1(t) = h(x(t)).
	\end{eqnarray*}
	Calculations then yield that,
	\begin{eqnarray*}
		b_1(x)
		& = & h(x) = x_1, \\
        Q_X & \neq & Q(\{b_1\}); \\
		    b_2(x)
		& = & L_f h(x) =  x_1 - x_2^2 + x_2, \\
        Q_X & \neq & Q(\{b_1,b_2\}); \\
		    b_3(x)
		& = & (L_f)^2 h =
		      x_1 -3 x_2^2 + 2 x_2, \\
		Q_X & = & Q(\{ b_1, b_2, b_3\}) = Q_{obs}(s); \\
		    n_0
		& = & 3 > 2 = n.
	\end{eqnarray*}
	The conclusion is that the observability index can be strictly
	higher than the state-space dimension of a rational system.
\end{example}
\begin{definition}\label{def:varietystructuredcanonicalformsiso}
Consider the set of rational systems
each of which is single-input-single-output and affine in the input, and
each of which is a minimal realization of its response map.
Define the
{\em variety-structured canonical form of such a rational system}
for the response-map equivalence relation
as the algebraic structure described by the equations,
\begin{eqnarray}
      X
  & = & \mathbb{R}^n, ~
	\mbox{or a hypersurface of $\mathbb{R}^{n+1}$,} \nonumber \\
      dx(t)/dt
  & = & f_0(x(t)) + f_1(x(t)) ~ u(t), ~
          x(0) = x_{0},  \\
      y(t)
  & = & h(x(t)), ~~ \\
  &   & f_0, ~ f_1 \in Q_{X,can}, ~ h \in Q_{X,can}. \nonumber
\end{eqnarray}
It could be that the functions $f$ and $h$ are partly structured
for which in each case a particular birational equivalence
form has to be formulated.
\end{definition}
\par
A canonical form of the set of minimal rational system
could also be formulated in an algebraic way by a subfield of the observation field.
A rational system is basically equivalent with its observation field.
It follows from \cite[Th. 6.1]{nemcova:schuppen:2009}
that for a response map there exists a rationally-observable realization
if and only if
the observation field of the response map has a finite transcedence degree.
Equivalently,
if for the observation field of the response map
there exists a finite set of generators.
A rationally-observable rational system realizing the response map
is therefore completely described by the observation field of its response map.
Because the observation field of the response map and the observation
field $Q_{obs}(s)$ of the rational system $s$ are isomorphic,
the above condition on the observation field of the response map
can be rewritten in terms of the observation field of the systems, $Q_{obs}(s)$.
This relation allows one then to formulate a canonical form
in terms of the observation field of a system
rather than in terms of the rational system representation.
\par
The algebraic formulation of an observable canonical form follows.
\begin{definition}
Consider the set of rational systems
each of which is single-input-single-output and affine in the input, and
each of which is a minimal realizations of its response map.
Define the {\em sequence of canonical observation subfields} of this system
on the basis of the observable canonical form by the equations,
	\begin{eqnarray}
		&   & \{ Q_{X, i} \subset Q_{obs}(s), ~ \forall ~ i \in \mathbb{Z}_n \},
		      \nonumber \\
		    Q_{X, i}
		& = & \{ r \in Q_{obs}(s) \subseteq Q_X | X_{i+1} = 0, \ldots, X_n = 0 \}, ~\\
		&   & \forall ~ i \in \mathbb{Z}_d, \nonumber \\
		    Q_{X, d}
		& = & Q_{obs}(s).
	\end{eqnarray}
\end{definition}
\par
This algebraic formulation of the observable canonical form
allows also an extension
to controlled-invariant observation subfields by state feedback.
This is then the analogon for rational systems
of the concept of a {\em dynamic cover} defined in \cite{wonham:morse:1972}.
This topic will be addressed in a future publication.
%
\section{THE THEOREM}\label{sec:theory}
\par
The reader finds in this section a proof that
the observable canonical form is a well-defined canonical form.
\begin{proposition}\label{proposition:obscanformminimalrealization}
	Consider a rational system
	in the observable canonical form of Def. \ref{def:rsystemobscanformsiso}.
	This system is:
	\begin{itemize}
		\item[(a)]
			rationally observable; and
		\item[(b)]
			a minimal realization of its response map.
	\end{itemize}
\end{proposition}
\begin{proof}
	(a) The proof is provided for the case $n=2$ from which the general case
	is easily deduced.
	The rational system in the observable canonical form is
	represented by the equations,
	\begin{eqnarray*}
                    dx_1(t)/dt
		& = & x_2(t) + f_1(x(t)) u(t), \\
		    dx_2(t)/dt
		& = & f_{2,0}(x(t)) + f_{2,1}(x(t)) u(t), \\
		    y(t)
		& = & x_1(t) = h(x(t)).
	\end{eqnarray*}
        The observation field is calculated according to,
	\begin{eqnarray*}
		    L_{f_{\alpha}}
		& = & [ x_2 + f_1(x) \alpha ]
		        \frac{\partial }{\partial x_1}
		      + [ f_{2,0}(x) + f_{2,1}(x) \alpha ]
		        \frac{\partial }{\partial x_2} ,  \\
		    Q_{obs}(s)
		& = & Q
		      \left(
		        \left\{
                        \begin{array}{l}
		           h,
				L_{f_{\alpha_k}} \ldots L_{f_{\alpha_1}} h, \\
                           \forall ~ \alpha_1, \ldots, \alpha_k \in U, ~
                           \forall ~ k \in \mathbb{Z}_+
			\end{array}
                        \right\}
                      \right), \\
		&   & h(x) = x_1 ~
                      \Rightarrow ~
                      x_1 \in Q_{obs}(s); ~
		      \forall ~ \alpha \in U, \\
		    L_{f_{\alpha}} h(x)
		& = & L_{f_{\alpha}} x_1 = [ x_2 + f_1(x) \alpha ]
		                  \frac{\partial x_1}{\partial x_1} + 0\\
		& = & x_2 + f_1(x) \alpha \in Q_{obs}(s); ~
                      \mbox{by assumption on $U$} ~ \\
                &   & \exists ~ \alpha_1, ~ \alpha_2 \in U, ~
                      \alpha_1 \neq \alpha_2, \\
		&   & x_2 + f_1(x) \alpha_1, ~ x_2 + f_1(x) \alpha_2
                      \in Q_{obs}(s), \\
		& \Rightarrow & f_1(x) (\alpha_1 - \alpha_2) \in Q_{obs}(s), \\
		& \Rightarrow & f_1(x) \in Q_{obs}(s), ~\\
                    x_2
		& = & [ x_2 + f_1(x) \alpha ] - f_1(x) \alpha \in Q_{obs}(s); \\
		    Q_X
		& = & Q(\{x_1, x_2\}) \subseteq Q_{obs}(s) \subseteq Q_X, \\
		    Q_X
		& = & Q_{obs}(s),
        \end{eqnarray*}
	hence the system is rationally observable.
        The proof shows that the input has to be varied to make the system
        rationally observable.
        With only a constant input, the system is not rationally observable.\\
	(b) The conclusion follows from (a), 
	the assumption of algebraic controllability 
	of Def. \ref{def:rsystemobscanformsiso}, and
	\cite[Proposition 6]{nemcova:schuppen:2010:aam}.
\end{proof}
\begin{proposition}\label{proposition:obscanformuniqueness}
        If two rational systems on a variety are
	both in the observable canonical form,
	of the same state-space dimension, and
        response-map equivalent and hence birationally equivalent,
        then they are identical.
\end{proposition}
\begin{proof}
	It follows from \cite[Th. 8]{nemcova:schuppen:2010:aam}
	that if two rational systems have the same response map
	and if they are both rationally observable and algebraically controllable
	then they are birationally related.
	\par
	The proof will be provided for the case of dimension $n=2$
	from which the general case follows by induction.
	Consider System 1 with state $x$ and System 2 with state $\overline{x}$
	both in the observable canonical form with the same state-set dimension.
	Assume that the systems are birationally related
	by the map $b: X \rightarrow \overline{X}$ hence
	$\overline{x}_1 = b_1(x_1,x_2)$ and $\overline{x}_2 = b_2(x_1,x_2)$.
	Then it follows from the observable canonical form for both systems that,
	\begin{eqnarray*}
		&   & x_1 = y = \overline{x}_1 = b_1 (x_1, x_2);\\
		    d \overline{x}_1(t)/dt
		& = & \overline{x}_2
		      + \overline{f}_{1}(\overline{x}) \alpha
		      = dx_1(t)/dt = x_2 + f_{1}(x) \alpha; \\
                    \overline{x}_2 - x_2
                & = & [ f_1(x) - \overline{f}_1(\overline{x}) ] \alpha; ~
                      \mbox{by assumption on $U$} ~\\
                &   & \exists ~ \alpha_1, ~ \alpha_2 \in U, ~
                      \alpha_1 \neq \alpha_2, \\
                    \overline{x}_2 - x_2
                & = & [ f_1(x) - \overline{f}_1(\overline{x}) ] \alpha_1
                      = [ f_1(x) - \overline{f}_1(\overline{x}) ] \alpha_2, \\
                    0
                & = & (\alpha_1 - \alpha_2) [ f_1(x) - \overline{f}_1(x) ], ~\\
                &   & \mbox{and, by} ~ (\alpha_1 - \alpha_2) \neq 0,\\
                    0
                & = & [ f_1(x) - \overline{f}_1(\overline{x}) ] ~
                      \Rightarrow ~ \overline{x}_2 - x_2 = 0, \\
		    x_2
		& = & \overline{x}_2 = b_2(x_1,x_2) ~
                      \Rightarrow ~ b(x) = x.
	\end{eqnarray*}
\end{proof}
\begin{proposition}\label{proposition:transformtoobscanform}
	Any minimal rational system $s$
	such that $Q_X \setminus Q_{obs}(s)$ is not algebraic over $Q_{obs}(s)$,
	can be transformed by a birational map
	to a rational system in the observable canonical form
	of Def. \ref{def:rsystemobscanformsiso}
	such that both systems are response-map equivalent.
\end{proposition}
\begin{proof}
	Consider a minimal rational system $s$ in the representation,
	\begin{eqnarray*}
		dx(t)/dt
		& = & f_0(x(t)) + f_1(x(t)) ~ u(t), \\
		    y(t)
		& = & h(x(t)), ~ x(t) \in \mathbb{R}^n.	
        \end{eqnarray*}
	Define the new state variable,
	\begin{eqnarray*}
		\overline{x}_1
		& = & h(x) = b_1(x); ~ \mbox{then,} \\
		    d \overline{x}_1(t)/dt
		& = & \frac{dh(x(t))}{dt} =
		      \sum_{i=1}^n [ f_{i,0}(x) + f_{i,1}(x) u] ~
                                   \frac{\partial h(x)}{\partial x_i}.
	\end{eqnarray*}
	Define the rational function and the variable,
	\begin{eqnarray*}
		    g_{1,1}(x)
		& = & \sum_{i=1}^n \frac{\partial h(x)}{\partial x_i}
			 f_{i,1}(x) \in Q_{obs}(s) = Q_{X}, \\
                    \overline{x}_2
                & = & \frac{d \overline{x}_1(t)}{dt}
		      - g_{1,1}(x(t)) ~ u(t) 
        \end{eqnarray*}
	\begin{eqnarray*}
                & = & \sum_{i=1}^n f_{i,0}(x)
                      \frac{\partial h(x)}{\partial x_i}
                      = b_2(x) \in Q_{obs}(s) = Q_X, \\
		    d \overline{x}_1 (t)/dt
		& = & \overline{x}_2 (t) + g_{1,1}(x(t)) u(t).
        \end{eqnarray*}
Because the system $s$ is a minimal rational system 
such that $Q_X \setminus Q_{obs}(s)$ are not algebraic over $Q_{obs}(s)$,
it is rationally observable, see \cite[Proposition 8]{nemcova:schuppen:2010:aam}.       
Note that the inclusions 
$\sum_{i=1}^n \frac{\partial h(x)}{\partial x_i} f_{i,1}(x) \in Q_{obs}(s)$ and 
$b_2(x) \in Q_{obs}(s)$ 
follow from the definition of the observation field of $s$ and 
from the assumption $\{0\} \subseteq U$. 
The equality $Q_{obs}(s) = Q_{X}$ follows from the rational observability of the system $s$.
By induction it follows that,
	\begin{eqnarray*}
		    \overline{x}_i
		& = & b_{i}(x) \in Q_X, \\
                    d \overline{x}_i(t)/dt
		& = & \overline{x}_{i+1}(t) + g_{i,1}(x(t)) ~ u(t), ~
		      i = 2, \ldots
        \end{eqnarray*}
It follows from Definition \ref{def:observabilityindex} and the discussion below it 
that there exists a finite observability index $n_o$ such that 
$b_1(x), \ldots, b_{n_o}(x)$ generate $Q_X$. 
Because $Q_X$ is the smallest set of rational functions on $X$ 
which distinguishes the points of $X$, 
its generators $b_1(x), \ldots, b_{n_o}(x)$
distinguish the points of $X$ as well. 
Therefore, the map
\[
	b: x \in X \mapsto ( b_1(x), \ldots, b_{n_o}(x))
\]
is injective and maps $X$ onto $b(X)$. 
Let us define 
\[
	\overline{X} = \zclosure (b(X)).
\]
Then the rational map $b = ( b_1, \ldots, b_{n_o}): X \rightarrow \overline{X}$ 
is invertible with the rational inverse $b^{-1}: \overline{X} \rightarrow X$. 
The fact that the components of $b^{-1}$ are elements of $Q_{\overline{X}}$
follows from the construction of $b$ 
(its components contain the components of $h$, $L_{f_{\alpha}}h$ and 
transcendence basis elements of $Q_{obs}(s)$),
see \cite[Proposition V.2]{nemcova:petreczky:schuppen:2016:cdc:obs} for the idea.
Hence, the varieties $X$ and $\overline{X}$ are birationally equivalent, i.e.
	\begin{eqnarray*}
		&   & \forall ~ x \in X \backslash A_{1,e}, ~
		      \overline{x} \in \overline{X} \backslash A_{2,e}, ~\\
		&   & b^{-1}( b(x)) = x, ~ b(b^{-1}(\overline{x})) = \overline{x}.
	\end{eqnarray*}
	\par
        Next, define the function,
\begin{eqnarray*}
	\overline{f}_{1,1}(\overline{x})
		& = & g_{1,1}(b^{-1}(\overline{x})) \in Q_{\overline{X}}, ~
		         \mbox{then,} \\
                    d \overline{x}_1(t)/dt
		& = & \overline{x}_2(t) + \overline{f}_{1,1}(x(t)) u(t).
\end{eqnarray*}
It follows by induction that,
\begin{eqnarray*}
	            \overline{f}_{i,1}(\overline{x})
		& = & g_{i,1}(b^{-1}(\overline{x})) \in Q_{\overline{X}}, \\
                    d \overline{x}_i(t)/dt
		& = & \overline{x}_{i+1}(t) + \overline{f}_{i,1}(x(t)) ~ u(t), \\
		    \overline{x}_i
		& = & b_{i}(x) \in Q_X, ~~
		      i = 1, \ldots n_o; \\
	            d \overline{x}_{n_o}(t)/dt
                & = & \sum_{i=1}^n [f_{i,0}(x) + f_{i,1}(x) u]
			\frac{\partial b_{n_o}(x)}{\partial x_i} \\
                & = & [ \sum_{i=1}^b f_{i,0}(x)
				     \frac{\partial b_{n_o}(x)}{\partial x_i} ] + \\
		&   & + [ \sum_{i=1}^n f_{i,1} (x)
				     \frac{\partial b_{n_o}(x)}{\partial x_i} ] u(t) \\
		& = & g_{n_o,0}(x(t)) + g_{n_o,1}(x(t)) ~ u(t), \\
		& = & \overline{f}_{n_o,0}(\overline{x}(t))
		      + \overline{f}_{n_o,1}(\overline{x}(t)) ~ u(t), ~
                      \mbox{where,} \\
		    \overline{f}_{n_o,0}(\overline{x})
		& = & g_{n_o,0}(b^{-1}(\overline{x})), ~ \\
		    \overline{f}_{n_o,1} (\overline{x})
		& = & g_{n_o,1}(b^{-1}(\overline{x})) \in Q_{\overline{X}}.
\end{eqnarray*}
Thus a rational system in the observable canonical form is obtained.
Then it follows with \cite[Def. 18, 19]{nemcova:schuppen:2010:aam}
that the original system and the constructed system are
birationally equivalent and hence response-map equivalent.
\end{proof}
\begin{theorem}\label{th:obscanform1}
	Consider the set of rational systems 
	which is such that for any system $s$ in this set it holds that
	$Q_X \backslash Q_{obs}(s)$ is not algebraic over $Q_{obs}(s)$.
	Restrict further attention to minimal rational systems
	and denote the resulting set of systems by $S_{rr,min}$.
	\par
	The observable canonical form of Def. \ref{def:rsystemobscanformsiso}
	for the specific set of minimal rational systems $S_{rr,min}$ and 
	for response-map equivalence,
	is a well-defined canonical form.
\end{theorem}
\begin{proof}
	The following three items prove the theorem.
	(1) Each rational system in the observable canonical
	form is a minimal realization of its response map.
	This follows from Proposition \ref{proposition:obscanformminimalrealization}.
	(2) Every rational system of a variety which is a minimal realization
	of its response map can be transformed to a rational system
	in the observable canonical form.
	This follows from Proposition \ref{proposition:transformtoobscanform}.
        (3) That a rational system can be transformed to
	a unique rational system in the observable canonical form.
	This condition, due to minimal realizations being birationally equivalent,
	is equivalent to proving that two rational systems in the observable
	canonical form which are birationally related, are identical.
	This follows from Proposition \ref{proposition:obscanformuniqueness}.
\end{proof}
%
\section{CONCLUDING REMARKS}\label{sec:concludingremarks}
Further research is needed
on the algebraic formulation of canonical forms,
on the control canonical form for the set of rational systems, and
on the use of the observable canonical form for system identification.
%
%
%
\section*{ACKNOWLEDGMENT}
The authors thank Prof. Eva Zerz (RWTH Aachen) 
for fruitful discussions and valuable comments on the paper. 
%


%

\end{document}